\documentclass[reqno]{amsart}
\usepackage{amssymb,amsthm,amscd}
\input{diagrams}

\newcommand{\fp}{{\mathfrak p}}
\newcommand{\fP}{{\mathfrak P}}
\newcommand{\fq}{{\mathfrak q}}

\newcommand{\Cl}{{\operatorname{Cl}}}
\newcommand{\Gal}{{\operatorname{Gal}}}

\newcommand{\Dv}{\operatorname{Div}}
\newcommand{\frb}[2]{\big[\frac{#1}{#2}\big]}

\newcommand{\Z}{{\mathbb Z}}
\newcommand{\Q}{{\mathbb Q}}

\newcommand{\C}{{\mathbb C}}

\newcommand{\lra}{\longrightarrow}

\newtheorem{thm}{Theorem}[section]
\newtheorem{prop}[thm]{Proposition}
\newtheorem{lem}[thm]{Lemma}
\newtheorem{cor}[thm]{Corollary}

\numberwithin{equation}{section}

\title{Harbingers of Artin's Reciprocity Law. \\
     I. The Continuing Story of Auxiliary Primes}
\author{F. Lemmermeyer}
\email{hb3@ix.urz.uni-heidelberg.de}
\address{M\"orikeweg 1, 73489 Jagstzell, Germany}

\begin{document}

\begin{abstract}
In this article we present the history of auxiliary primes used in
proofs of reciprocity laws from the quadratic to Artin's reciprocity
law. We also show that the gap in Legendre's proof can be closed with
a simple application of Gauss's genus theory.
\end{abstract}

\maketitle
\markboth{Harbingers of Artin's Reciprocity Law}
         {\today \hfil Franz Lemmermeyer}
\bigskip

\section*{Introduction}

Artin's reciprocity law, the central result in class field theory,  
is an isomorphism between some ideal class group and a Galois group, 
and sends the class of a prime ideal $\fp$ to the Frobenius automorphism 
of $\fp$. In this series of articles we would like to discuss results
that, in the long run, turned out to be related to certain pieces of 
the quite involved proof of Artin's reciprocity law. 

In this article we sketch the history of auxiliary primes. These are
prime numbers whose existence was needed in various proofs of 
the quadratic as well as the higher reciprocity laws from Kummer
to Artin. Subsequent articles will discuss the irreducibility
of the cyclotomic equation, Gauss's Lemma and the transfer map,
and finally Bernstein's reciprocity law, which is the special
case of Artin's reciprocity law for unramified abelian extensions
stated by Bernstein long before Artin.

Auxiliary primes in connection with proofs of the quadratic
reciprocity law have a long history. Legendre's attempt at 
proving the reciprocity law failed (see \cite[Ch. 1]{LRL}) 
because he could not guarantee the existence of certain primes.
Gauss's first proof via induction used auxiliary primes whose
existence was secured with the help of an ingenious elementary
argument. After Kummer had proved the $p$-th power reciprocity
law in the fields of $p$-th roots of unity, where $p$ is a 
regular prime, he gave a proof of the quadratic reciprocity
law based on analogous principles which also required auxiliary 
primes; their existence, Kummer claimed, would follow easily from 
the analytic techniques of Dirichlet.

\section{Legendre}

Legendre attempted to prove the quadratic reciprocity law with the
help of his theorem on the solvability of the equation 
$ax^2 + by^2 = cz^2$, whose proof relied heavily on the technique 
of reduction provided by Lagrange:

\begin{thm}\label{LT}
Assume that the integers $a, b, c$ are positive, pairwise coprime,
and squarefree. Then the diophantine equation
\begin{eqnarray}\label{LE} ax^2+by^2+cz^2 = 0 \end{eqnarray}
has non-trivial solutions in $\Z$ if the following three conditions
are satisfied:
$bc$, $ac$ and $-ab$ are squares modulo $a$, $b$ and $c$,
respectively.
\end{thm}

Let us now briefly explain Legendre's idea for proving the
quadrat reciprocity law
$$ \Big( \frac pq \Big) = (-1)^{\frac{p-1}2 \cdot \frac{q-1}s} \Big( \frac qp \Big) $$
for distinct odd prime numbers $p$. Like Gauss, Legendre distinguished 
eight cases; he denoted primes $\equiv 1 \bmod 4$ by $a$ and $A$,
and primes $\equiv 3 \bmod 4$ by $b$ and $B$. The first of Legendre's
eight cases then states that
$$ \Big(\frac ba \Big) = 1  \quad \Longrightarrow \quad
   \Big(\frac ab \Big) = 1. $$
Legendre observes that the equation $x^2 + ay^2 = bz^2$
is impossible modulo $4$; thus the conditions in Thm. \ref{LT} 
cannot all be satisfied. But there are only two conditions,
one of which, namely $(\frac ba) = +1$, holds by assumption.
Thus the second condition $(\frac{-a}b) = +1$ must fail; but 
then $(\frac{-a}b) = -1$ implies $(\frac ab) = +1$ as desired.

Next consider the equation $Ax^2 + ay^2 = bz^2$; it is easily seen
to be impossible modulo $4$, hence at least one of the three conditions 
in Thm. \ref{LT}, $(\frac{-aA}b) = +1$, $(\frac{Ab}a) = +1$, and
$(\frac{ab}A) = +1$, must fail. Now choose an auxiliary prime $b$ 
such that $(\frac ba) = 1$ and $(\frac Ab) = -1$. By what Legendre 
has already proved, this implies $(\frac ab) = 1$ and $(\frac bA) = -1$. 
Then the first condition $(\frac{-aA}b) = +1$ always holds, and the last 
two conditions are equivalent to  $(\frac{A}a) = +1$ and $(\frac{a}A) = -1$.
Since these cannot hold simultaneously, Legendre concludes that
$$ \Big(\frac aA \Big) = 1  \quad \Longrightarrow \quad
   \Big(\frac Aa \Big) = 1, $$
as well as
$$ \Big(\frac Aa \Big) = -1  \quad \Longrightarrow \quad
   \Big(\frac aA \Big) = -1. $$
This result is proved under the assumption that the following lemma holds:

\begin{lem}\label{LLeg1}
Given distinct primes $a \equiv A \equiv 1 \bmod 4$, there exists a
prime $b$ satisfying the conditions $b \equiv 3 \bmod 4$, 
$(\frac ba) = 1$ and $(\frac Ab) = -1$. 
\end{lem}

As for the existence of these primes $b$, Legendre remarks
\begin{quote}
On peut s'assurer qu'il y en a une infinit\'e; mais voici une
d\'emon\-stra\-tion directe qui \'ecarte toute difficult\'e.\footnote{It
is possible to make sure that there are infinitely many such primes;
but here is a direct proof that avoids any difficulty.} 
\end{quote}
He then gives a second proof of this claim, and this time admits
\begin{quote}
Dans cette d\'emonstration, nous avons suppos\'e seulement  qu'il
y avois un nombre premier $b$ de la form $4n-1$, qui pouvoit diviser
la formule $x^2 + Ay^2$.\footnote{In this proof we have only assumed
that there be a prime number $b$ of the form $4n-1$ which divides
the form $x^2 + Ay^2$.}
\end{quote}

In fact, an equation of the form $x^2 + Ay^2 = abz^2$ is impossible
in nonzero integers since $x^2 + y^2 \equiv -z^2 \bmod 4$ implies
that $x$, $y$ and $z$ must be even. Thus the conditions in 
Legendre's Theorem cannot be satisfied. If we choose $b$ prime
such that $(\frac Ab) = -1$, then by what we have already proved
we know that $(\frac bA) = -1$. Since $-A$ is a square modulo 
$a$ and $b$, the last condition that $ab$ be a square modulo $A$
cannot hold; thus $(\frac aA) (\frac bA) = -1$, and this implies
the claim $(\frac aA) = +1$.

For future reference, let us state Legendre's assumption explicitly 
as 

\begin{lem}\label{LLeg2}
For each prime $p \equiv 1 \bmod 4$ there exists a prime
$q \equiv 3 \bmod 4$ such that $(p/q) = -1$.
\end{lem}

Finally, for proving that
$$ \Big(\frac bB \Big) = 1  \quad \Longrightarrow \quad
   \Big(\frac Bb \Big) = -1, $$
Legendre assumes the following

\begin{lem}\label{LLeg3}
For primes $q \equiv r \equiv 3 \bmod 4$ there exists a prime
$p \equiv 1 \bmod 4$ such that $(p/q) = (p/r) = -1$.
\end{lem}

Legendre returns to the problem of the existence of his auxiliary
primes at the end of his memoir \cite{Leg85}: on p. 552, he writes
\begin{quote}
Il seroit peut-\^etre n\'ecessaire de d\'emontrer rigoureusement
une chose que nous avons suppos\'ee dans plusieurs endroits de cet
article, savoir, qu'il y a une infinit\'e de nombres premiers compris
dans toute progression arithm\'etique, dont le premier terme \& la
raison sont premier entr'eux \ldots\footnote{It is perhaps necessary
to give a rigorous proof of something that we have assumed in several
places of this article, namely, that there is an infinite number of
primes contained in every arithmetic progression whose first term
and ratio are coprime \ldots}.
\end{quote}
He then sketches an approach to this result which he says is
too long to be given in detail.

This last remark is problematic in more than one way: Legendre's
Lemma \ref{LLeg3} certainly follows from the theorem on primes in 
arithmetic progressions: take a quadratic nonresidues $m \bmod q$
and $n \bmod r$, and use the Chinese Remainder Theorem\footnote{In
Legendre's times, this result was credited to Bachet; the name
Chinese Remainder Theorem became common only much later.} to
find an integer $c$ with $c \equiv 3 \bmod 4$, $c \equiv m \bmod q$
and $c  \equiv n \bmod r$. Then any prime in the progression $c + 4qrk$
will satisfy Lemma \ref{LLeg3}. 

On the other hand it is not clear at all whether Lemma \ref{LLeg1} 
and  Lemma \ref{LLeg2} can be deduced from the theorem on primes in 
arithmetic progression without assuming the quadratic reciprocity
law: this was already observed by Gauss in \cite[p. 449]{GDA}.

In the second version of his proof in \cite{LegTN}, Legendre managed 
to do with the existence result in Lemma \ref{LLeg2} alone.

\section{Gauss}

Gauss's first proof of the quadratic reciprocity law used auxiliary
primes that were quite similar to those used by Legendre. Below we 
will first sketch the main idea behind Gauss's first proof and then
use the techniques from his second proof (genus theory of binary
quadratic forms) to prove Legendre's Lemma \ref{LLeg2}.

\subsection*{Gauss's First Proof}
The first proof of the quadratic reciprocity law found by Gauss was 
based on induction (or descent, as Fermat would have said). Assuming
that the reciprocity law holds for small primes, Gauss assumes that
$(\frac rq) = +1$, where $r$ and $q \equiv 1  \bmod 4$ are primea;
then $x^2 \equiv r \bmod q$ shows that there is an equation of the
form $e^2 = r + fq$. Choosing $e$ and $f$ carefully and reducing the
equation modulo $p$, Gauss is able to prove that $(q/r) = +1$.

This procedure fails if $(r/q) = -1$ because in this case we do not
have an equation like $e^2 = r + fq$ to work with. To get around
this problem, Gauss used an auxiliary prime: he found an
elementary, but highly ingenious proof of the following

\begin{lem}\label{LGp}
For every prime $q \equiv 1 \bmod 4$ there is a prime $p < q$
with $(q/p) = -1$.
\end{lem}

By what he already proved, Gauss concluded that $(p/q) = -1$ as
well, hence $pr$ is a quadratic residue modulo $q$; by carefully 
studying the equation $e^2 = pr + fq$ and reducing it modulo
various primes, Gauss finally proves that $(q/r) = -1$.

The proof of Lemma \ref{LGp} is trivial if one is allowed to use
quadratic reciprocity: let $a < q$ be an odd quadratic nonresidue 
modulo $q$; then there must be a prime number $p \mid a$ with 
$(p/q) = -1$. Quadratic reciprocity takes care of inverting the
symbol. If $p \equiv 1 \bmod 8$, then $2$ is a quadratic residue
modulo $p$, and proving the existence of $a$ is easy (in this case, 
$a$ actually can always be chosen $< \sqrt{q}$). If $p \equiv 5 \bmod 8$,
then $a = \frac{p+1}2$ will work.

\subsection*{A Proof of Legendre's Lemma}
Let us now give a proof of Legendre's Lemma \ref{LLeg2} based
on Gauss's genus theory. We start by recalling the basic results.
The equivalence classes of primitive binary quadratic forms with 
nonsquare discriminant $\Delta$ form a group called the class group,
which, for fundamental discriminants (these are discriminants of
quadratic number fields) is isomorphic to the class group in the
strict sense of $\Q(\sqrt{\Delta}\,)$. 
Every form $Q = (A,B,C) = Ax^2 + Bxy + Cy^2$ represents integers $a$
coprime to $\Delta$, i.e., there exist integers $x, y$ such that 
$Q(x,y) = a$. Fundamental discriminants $\Delta$ can be written
uniquely as a product $\Delta = \Delta_1 \cdots \Delta_t$ of prime 
discriminants $\Delta_j$; the value $\chi_j(a) = (\Delta_j/a)$ is
well defined and does not depend on the choice of $a$ or on the
representative of the equivalence class of $Q$. This allows us to
define $\chi_j(Q) = \chi_j(a)$, where $a$ is any integer coprime to 
$\Delta$ represented by $Q$. The forms for which
all characters $\chi_j$ are trivial form a subgroup of the class group
called the principal genus, which consists of all square classes.
The main result in genus theory is

\begin{thm}
Given any set of integers $c_1, \ldots, c_t$, there exists a primitive
form $Q$ with discriminant $\Delta =  \Delta_1 \cdots \Delta_t$ such that 
$$ \chi_1(Q) = (-1)^{c_1}, \ldots, \chi_t(Q) = (-1)^{c_t} $$
if and only if $c_1 + \ldots + c_r$ is even.
\end{thm}

Given a prime $p \equiv 1 \bmod 4$
we have to find a prime $q \equiv 3 \bmod 4$ such that 
$(p/q) = -1$. The last condition is equivalent to $(-p/q) = +1$.

If $p \equiv 1 \bmod 4$, then $Q = (2,2,\frac{p+1}2)$ is primitive
and has discriminant $\Delta = -4p$. If $p \equiv 5 \bmod 8$, then
$\frac{p+1}2 \equiv 3 \bmod 4$, and the form $Q$ represents integers 
$\equiv 3 \bmod 4$ since e.g. $Q(0,1) = \frac{p+1}2$. This implies that
$Q(0,1)$ must have a prime divisor $q \equiv 3 \bmod 4$. Now we observe

\begin{lem}
Let $Q$ be a primitive binary quadratic form with discriminant $\Delta$.
If a prime number $q \nmid \Delta$ divides $Q(x,y)$ for some coprime 
integers $x, y$, then $(\frac{\Delta}q) = +1$.
\end{lem}

\begin{proof}
Assume that $Q = (A,B,C)$ is a primitive form with discriminant
$\Delta = B^2 - 4AC$; if $Ax^2 + Bxy + Cy^2 \equiv 0 \bmod q$, then 
$$ 4A(Ax^2 + Bxy + Cy^2) = (2Ax+ By)^2 - \Delta x^2 \equiv 0 \bmod q. $$
If $q \mid x$, then $q \mid C$ since $\gcd(x,y) = 1$; but then
$\Delta = B^2 - 4AC \equiv B^2 \bmod C$.
\end{proof}

If $p \equiv 1 \bmod 8$, then $Q$ only represents integers 
that are even or $\equiv 1 \bmod 4$. But in this case, the 
class number of forms with discriminant $-p$ is divisible 
by $4$, and it is easy to see that $[Q] = 2[Q_1]$ for some 
form $Q_1$. If $Q_1$ represents an integer $\equiv 3 \bmod 4$, 
then we are done. If the odd integers represented by $Q_1$ 
are all $\equiv 1 \bmod 4$, then $Q_1$ is in the principal genus:

\begin{lem}
Let $p \equiv 1 \bmod 4$ be a prime number and let $Q$ be a primitive
binary quadratic form with discriminant $-4p$ that represents only
integers that are even or $\equiv 1 \bmod 4$. Then $Q$ is in the
principal genus.
\end{lem}

\begin{proof}
Since $\Delta = -4 \cdot p$ is a product of two prime discriminants,
there are only two genus characters, namely $\chi_4(a) = (\frac{-1}a)$
and $\chi_p(a) = (\frac pa)$. If an integer $a$ coprime to $\Delta$ 
is represented by $Q$, then $\chi_4(a) = \chi_p(a)$, and $Q$ is in the 
principal genus if and only if $\chi_1(a) = +1$.

Thus if the odd integers represented by $Q$ are all $\equiv 1 \bmod 4$, 
the form $Q$ must lie in the principal genus.
\end{proof}

Continuing in this way we see that some form $Q_j$ with discriminant
$-4p$ must represent an integer $\equiv 3 \bmod 4$ (unless the $2$-class
group of primitive forms with discriminant $-4p$ has an infinitely
divisible element -- but this would contradict the finiteness of the class
number). This integer must have a prime factor $q \equiv 3 \bmod 4$, 
and since $q$ is represented by some form with discriminant $-4p$, we must
have $(-p/q) = +1$. Since $(-1/q) = -1$, Lemma \ref{LLeg2} follows.

The proof of Legendre's Lemma given above closely follows the way
it was discovered. Actually we can give a much shorter proof: let
$Q$ be a primitive form with discriminant $\Delta = -4p$ and with
genus characters $\chi_1(Q) = \chi_2(Q) = -1$. Then any odd integer
represented by $Q$ is $\equiv 3 \bmod 4$, and now our claim
follows immediately.

It is clear that this proof can be generalized considerably. I have 
meanwhile found that Lubelski \cite{Lub} used similar methods for
proving the following result:

\begin{prop}
Every polynomial $f(x) = ax^2 + bx + c \in \Z[x]$, where $d = b^2 - 4ac$
is not a negative square, has infinitely many prime divisors 
$q \equiv -1 \bmod 4$.
\end{prop}

\section{Dirichlet}

Dirichlet, who had studied mathematics in Paris, was not only familiar
with the work of Euler and Gauss, but also knew the results due to
Lagrange and Legendre\footnote{Dirichlet also was aware of Lambert's
article on the irrationality of $\pi$; Lambert's proof was streamlined
by Legendre and later led to dramatic progress under the hands of 
Hermite. In Germany, Lambert's work was apparently less appreciated.}.
In particular he knew Legendre's conjecture about primes in arithmetic 
progressions, and also was aware of the fact that the proof Legendre had 
given was incomplete; in \cite{Dir}, he wrote:
\begin{quote}
Legendre bases the theorem he wants to prove on the problem of 
finding the largest number of consecutive terms of an arithmetic
progression that are divisible by one of a set of given primes; however,
he solves this problem only by induction. If one tries to prove the
solution of the maximum problem found by him, whose form is most
remarkable because of its simplicity, then one comes across
some big difficulties which I did not succeed in surmounting.
\end{quote} 

Dirichlet therefore looked for different tools and found insoiration
in the work of Euler, who had proved the infinitude of primes using
the divergence of the zeta function at $s = 1$. Using Euler's method
Dirichlet was able to prove the following

\begin{thm}\label{TPAP}
Let $a$ and $m$ be coprime natural numbers. Then there exist
infinitely many primes $p \equiv a \bmod m$.
\end{thm}

For proving the result Dirichlet introduced characters
$\chi: (\Z/m\Z)^\times \lra \C^\times$ and their associated
L-series $L(s,\chi) = \sum_{n \ge 1}  \chi(n) n^{-s}$. For real
$s > 1$, these L-series can be written as an Euler product:
$L(s,\chi) = \prod_{p} (1 - \chi(p) p^{-s})^{-1}$ if one carefully
defines the values of $\chi(p)$ for primes $p \mid m$. Theorem
\ref{TPAP} is a direct consequence of the following result:

\begin{thm}
If $\chi$ is a nontrivial Dirichlet character modulo $m$, then 
$$ \lim_{s \to 1} L(s,\chi) \ne 0. $$
\end{thm}

Dirichlet never worked with the value $L(1,\chi)$; it was Mertens
who first proved that $L(s,\chi)$ converges at $s = 1$ for all
nontrivial Dirichlet characters $\chi$. Below we will use
$L(1,\chi)$ as an abbreviation of $\lim_{s \to 1} L(s,\chi)$.

For prime values $p$, Dirichlet could express $L(1,\chi)$ as a
finite sum and directly prove that it does not vanish. For composite
values of $m$, on the other hand, he had a rather technical proof
that he abandoned in favor of a more conceptual approach. It is
rather easy to show that  $L(1,\chi) \ne 0$ if $\chi$ assumes nonreal
values, so the problem is showing $L(1,\chi) \ne 0$ for all real
Dirichlet characters. Dirichlet found that the values  $L(1,\chi)$ 
were connected to arithmetic properties of quadratic number fields.
For making the connection, Dirichlet had to show 

\begin{lem}[Dirichlet's Lemma]
If $\chi$ is a primitive quadratic Dirichlet character defined modulo $d$, 
then $\chi(a) = (d/a)$ for all positive integers $a$ coprime to $D$.
\end{lem}

Dirichlet's Lemma is essentially Euler's version of quadratic reciprocity:
since $\chi$ is defined modulo $d$, Dirichlet's Lemma implies that
$(d/p) = (d/q)$ for positive prime numbers $p \equiv q \bmod d$.

We have already remarked that it was Mertens who first succeeded in 
proving Dirichlet's Theorem without using Dirichlet's Lemma; there
are reasons to believe that Mertens' proof is similar to the one that
Dirichlet abandoned since it is in fact based on Dirichlet's results
on the asymptotic behaviour of the divisor function.

We have also observed that Dirichlet's Theorem \ref{TPAP} implies 
Legendre's Lemma \ref{LLeg3}, but that it cannot be used for 
deriving Lemma \ref{LLeg2}. The last lemma would follow from a 
result that Dirichlet announced but whose proof he only sketched 
in a very special case:

\begin{thm}\label{TDQF}
Let $Q = (A,B,C)$ be a primitive quadratic form with discriminant
$\Delta = B^2 - 4AC$. Then the set $S_Q$ of primes represented
by $Q$ has Dirichlet density
$$ \delta(S_Q) = \begin{cases}
                  \frac1h    & \text{ if }\ Q \not\sim (A,-B,C), \\
                  \frac1{2h} & \text{ if }\ Q \sim (A,-B,C),
                 \end{cases} $$
where $h$ is the class number of forms of discriminant $\Delta$.
\end{thm}

Dirichlet's claims were slightly different, since he
worked with forms $(A,2B,C)$ whose middle coefficients are even.
If $Q = (1,0,1)$, then $h = 1$, hence Thm. \ref{TDQF} tells us
that primes represented by $Q$ have Dirichlet density $\frac12$.

If $\Delta = -23$, then $h = 3$, and the primes represented by the 
principal form $(1,1,6)$ have density $\frac16$, whereas the
primes represented by $(2,1,3)$ have density $\frac13$. The forms
$(2,3,4)$ and $(2,-1,3)$ clearly represent the same primes, and
primes represented by any form with discriminant $\Delta$ have
density $\frac13 + \frac16 = \frac12$: these are the primes $p$
satisfying $(-23/p) = +1$. In general, Dirichlet's Theorem \ref{TDQF}
has the following

\begin{cor}\label{CDir}
Let $\Delta$ be a quadratic discriminant. Then the prime numbers $p$ 
with $(\Delta/p) = +1$ and those with $(\Delta/p) = -1$ each
have Dirichlet density $\frac12$.
\end{cor}

This corollary, by the way, can be proved directly without having
to assume the quadratic reciprocity law; in fact, Dirichlet's Lemma
has to be invoked only for transforming a quadratic Dirichlet character 
$\chi$ modulo $\Delta$ into a Kronecker character $(\Delta/\cdot)$, and the 
character in Cor. \ref{CDir} is already a Kronecker character.

We now apply Cor. \ref{CDir} to forms with discriminant $\Delta = -p$,
where $p \equiv 1 \bmod 4$ is a prime number. In this case there are 
exactly two forms with discriminant $\Delta = -4p$ satisfying 
$(A,B,C) \sim (A,-B,C)$, namely the principal form $Q_0 = (1,0,p)$ 
and the ambiguous form $Q_1 = (2,2,\frac{p+1}2)$. 

If $p \equiv 5 \bmod 8$, then the number $h_1 = h/2$ of classes in the
principal genus is odd, and the primes represented by forms in the 
principal genus have density 
$\frac1{2h} + \frac{h_1-1}{2h} = \frac14$. 

If $p \equiv 1 \bmod 8$, then the number $h_1 = h/2$ of classes in the
principal genus is even, and the primes represented by forms in the 
principal genus have density 
$\frac{2}{2h} + \frac{h_1-2}{2h} = \frac14$. In particular, the
primes represented by some form not in the principal genus have
density $\frac14$, and in particular, we find the following strengthening
of Legendre's Lemma \ref{LLeg2}: the primes $q$ with 
$(\frac{-1}q) = (\frac pq) = -1$ have Dirichlet density $\frac14$.
 
\section{Kummer and Hilbert}

Starting with Kummer, auxiliary primes became an indispensible
tool for proving higher reciprocity laws, and variants of the
corresponding existence results can be found in the work of 
Hilbert, Furtw\"angler and Takagi on class field theory.

\subsection*{Kummer's Proof of Quadratic Reciprocity}

In \cite{KuQR}, Kummer gave two new proofs of the $p$-th power 
reciprocity law for regular primes; in the introduction, he 
proves the quadratic reciprocity law by similar methods. The
first proof is similar to Legendre's; the main difference is
that Kummer replaces Legendre's equation $ax^2 + by^2 = cz^2$
by the Pell equation. In the course of his proof, Kummer assumes
the following existence results:

\begin{lem}
Let $p \equiv 3 \bmod 4$ and $q \equiv 1 \bmod 4$ be distinct primes 
numbers. Then there exist primes $p' \equiv 3 \bmod 4$ such that 
$(p'/p) = (p'/q) = -1$.
\end{lem}

This clearly follows from Dirichlet's theorem on primes in 
arithmetic progression.

\begin{lem}
Let $p \equiv 3 \bmod 4$ and $q \equiv 1 \bmod 4$ be distinct primes 
numbers. Then there exist primes $p' \equiv 3 \bmod 4$ such that 
$(p/p') = (q/p') = -1$.
\end{lem}

This lemma can be proved in the same way as Legendre's Lemma \ref{LLeg2}
above by exploiting the genus theory for forms with discriminant
$\Delta = -4pq$, and also follows from Dirichlet's Theorem \ref{TDQF}.

\begin{lem}
Given primes $q \equiv q' \equiv 1 \bmod 4$ be distinct primes 
numbers. Then there exist primes $p \equiv p' \equiv 3 \bmod 4$
such that $(p/q) = -1$, $(p/q') = +1$, $(p'/q) = +1$ and 
$(p'/q') = -1$. 
\end{lem}

This is again a consequence of Dirichlet's theorem on primes in 
arithmetic progression.

Kummer does not give proofs for these existence results and simply
remarks that they can be derived using Dirichlet's methods. It must 
be observed, however, that Dirichlet's proof of the theorem on primes
in arithmetic progression used the quadratic reciprocity law for
showing that every real primitive Dirichlet character $\chi$ modulo $D$ 
has the form $\chi(p) = (D/p)$ for some quadratic discriminant $D$; thus
things are not as easy as Kummer pretends. Dirichlet actually gave a 
proof that there exist infinitely many primes in coprime residue 
classes modulo $p$, where $p$ is a prime number, that did not depend 
on quadratic reciprocity. With a little effort, this proof can be 
extended to residue classes modulo $4pq$, which would cover the 
existence results that Kummer had used.

\subsection*{Kummer's Reciprocity Law}

After Gauss and Jacobi had stated and proved the first ``higher''
reciprocity laws for fourth and third powers, it became clear that
even for stating any $p$-th power reciprocity law one had to
work in the field of $p$-th roots of unity. The problems coming
from nonunique factorization were overcome by Kummer's invention
of ideal numbers. Eisenstein used Kummer's ideal numbers for
proving ``Eisenstein's reciprocity law'', which turned out to
be as indispensible for proving higher reciprocity laws as the existence
of auxiliary primes; only Artin finally succeeded in eliminating
Eisenstein's reciprocity law by replacing it with the technique of
abelian twists used by Chebotarev.

Kummer's existence result was stated and proved in \cite[p. 138]{KuAR}.
In its original form, it reads as follows:

\begin{thm}
Let $F(\alpha)$, $F_1(\alpha)$, $F_2(\alpha)$, \ldots, $F_{n-1}(\alpha)$
denote real complex numbers satisfying the condition that the product
$$ F(\alpha)^m F_1(\alpha)^{m_1} F_2(\alpha)^{m_2} 
                      \cdots F_{n-1}(\alpha)^{m_{n-1}} $$ 
for integral values of the exponents becomes a $\lambda$-th power
if and only if all these exponents are $\equiv 0 \bmod \lambda$.
Then there exist infinitely many prime numbers $\phi(\alpha)$ with
respect to which the indices of the complex numbers 
$F(\alpha)$, $F_1(\alpha)$, \ldots, $F_{n-1}(\alpha)$ are proportional
modulo $\lambda$ to arbitrarily given numbers.
\end{thm}

In Kummer's terminology, $\alpha$ is a primitive $\lambda$-th root
of unity. Kummer distinguished between ideal complex numbers (roughly
corresponding to Dedekind's ideals; see \cite{LJac} for a more
precise explanation of ideal numbers) and real complex numbers
(``wirkliche komplexe Zahlen'', that is, elements of the ring $\Z[\alpha]$),
which he denoted by $F(\alpha)$ even if the numbers were ideal.
The index of $F(\alpha)$ with respect to $\phi(\alpha)$ is the
integer $c$ for which
$$ \Big(\frac{F(\alpha)}{\phi(\alpha)} \Big) = \alpha^c, $$
where the symbol on the left is the $\lambda$-th power residue
symbol in $\Z[\alpha]$ and $c$ is determined modulo $\lambda$.

Translated into modern terms, Kummer's theorem becomes

\begin{thm}\label{TKum}
Let $\zeta$ be a primitive $p$-th root of unity. If
$\alpha_1$, \ldots, $\alpha_r \in \Z[\zeta]$ are independent modulo
$p$-th powers, then for any set of integers $c_1, \ldots, c_r$ there 
is an integer $m$ not divisible by $p$ and infinitely many prime 
ideals $\fp$ in $\Z[\zeta]$ such that 
$$ \Big(\frac{\alpha_1}{\fp}\Big)^m = \zeta^{c_1}, \ldots,
   \Big(\frac{\alpha_r}{\fp}\Big)^m = \zeta^{c_r}. $$
\end{thm}

The main idea behind Kummer's result is applying Dirichlet's method
to the L-series $L(s,\chi)$, where the $\chi$ are $p$-th power residue
characters. The nonvanishing of $L(1,\chi)$ comes from the analytic
class number formula (which essentially shows that 
$\lim_{s \to 1} (s-1)\zeta_K(s)$ is finite and nonzero) and the
factorization $\zeta_K(s) = \prod_j L(s,\chi^j)$ of ``Dedekind's''
zeta function as a product of L-series. 

\subsection*{Hilbert} 
Hilbert proved the nonvanishing of the L-series needed for proving 
Theorem \ref{TKum} in \cite[Hs. 27]{Hil} and Kummer's theorem in 
\cite[Satz 152]{Hil} of his Zahlbericht. Later, Hilbert worked out
the theory of quadratic extensions of number fields with odd class
number by closely following Kummer's work on Kummer extensions of 
the fields of $p$-th roots of unity for regular primes $p$. In 
particular, Hilbert needed the following existence theorem, which
can be found in \cite[Satz 18]{Hil2}:

\begin{thm}\label{THil2}
Let $K$ be a number field, and assume that 
$\alpha_1, \ldots, \alpha_s \in K^\times$ are independent modulo squares.
Then for any choice of $c_1, \ldots, c_s \in \{\pm 1\}$, there
exist infinitely many prime ideals $\fp$ in $K$ with
$$ \Big( \frac{\alpha_1}{\fp} \Big) = c_1, \quad \ldots, \quad
   \Big( \frac{\alpha_s}{\fp} \Big) = c_s. $$
\end{thm}

The nonvanishing of the corresponding L-series is an immediate
consequence of Dedekind's class number formula, according to which
$\lim_{s \to 1} \zeta_K(s) \ne 0$. The special case in which $K$ is a 
multiquadratic extension of $\Q$ is also the result that Kummer
alluded to when he stated that the existence of his auxiliary 
primes can be verified using Dirichlet's methods. One should, 
however, bear in mind that Kummer did not have the theory of ideals
(or ideal numbers) in quadratic number fields at his disposal; 
this was worked out later by Dedekind.

In the special case $K = \Q$, $\alpha_1 = -1$ and $\alpha_2 = p$,
Thm. \ref{THil2} is exactly Legendre's Lemma \ref{LLeg2}.

Similar results were used by Furtw\"angler\footnote{In \cite[Satz 17]{Fu2},
Furtw\"angler sketches the proof, in most other places he simply 
refers to Hilbert's Bericht.} and Takagi\footnote{See in particular
\cite[\S\ 3]{Tak}.} in their proofs of the main results of class field 
theory.

\section{Kronecker, Frobenius, Chebotarev}

The simplest case of Kummer's existence theorem claims that
in Kummer extensions $L/K$ of $K = \Q(\zeta_p)$, there are
infinitely many prime ideals that split completely and infinitely
many prime ideals that remain inert. A more general result (although
initially restricted to the case where the base field is $\Q$) was 
stated by Kronecker:

\begin{thm}\label{TKron}
Let $L$ be a number field. Then there exist infinitely many prime
ideals of inertia degree $1$ in $L$. More exactly, the Dirichlet density
of such primes is $1/(N:\Q)$, where $N$ is the normal closure of
$L/\Q$.
\end{thm}

Applied to quadratic extensions $L = \Q(\sqrt{m}\,)$, this implies 
the existence of infinitely many primes $p$ with $(\frac mp) = +1$,
and in fact half the primes have this property; consequently, there
also must be infinitely many primes with $(\frac mp) = -1$.

Similarly, applied to  $L = \Q(\zeta_m)$ Kronecker's density theorem
predicts the existence of infinitely many primes $p \equiv 1 \bmod m$. 
Both of these special cases can actually be proved by elementary means

Kronecker's density theorem was generalized immediately by 
Frobenius \cite{Fro}. The following special case of Frobenius'
result can be stated in the same language as Kronecker's:

\begin{thm}[Frobenius Density Theorem for Abelian Extensions]\label{TFrAb}
Let $L/K$ be an abelian extension, and $F$ an intermediate field
such that $L/F$ is cyclic. Then the set $S_F$ of prime ideals with
decomposition field $F$ has Dirichlet density
$$ \delta(S_F) = \frac{\phi(L:F)}{(L:K)}. $$
\end{thm}

The main idea behind this special case is again Dedekind's density
theorem. Applied to the biquadratic number field 
$L = \Q(\sqrt{-1},\sqrt{p}\,)$ and the subextension
$F = \Q(\sqrt{-p}\,)$, Thm. \ref{TFrAb} guarantees the existence
of infinitely many primes $q$ with $(\frac qp) = +1$ and 
$(\frac{-1}q) = (\frac pq) = -1$. In particular, it implies
Legendre's Lemma \ref{LLeg2}.

For generalizing his result to general normal extensions, Frobenius 
had to introduce the notion of a division: the division $\Dv(\phi)$ 
of an element $\phi \in G$ is the set of all $\sigma \in G$ with the 
property that $\sigma = \tau^{-1} \phi^k \tau$ for some $\tau \in G$ 
and an exponent $k$ coprime to the order of $\sigma$. Modulo some
group theoretical preliminaries, the following result is then 
rather easily proved:

\begin{thm}[Frobenius Density Theorem]\label{TFD}
Let $L/K$ be a normal extension, and let $D$ be a division in 
$G = \Gal(L/K)$. Let $S$ denote the set of unramified prime 
ideals $\fp$ in $K$ with the property that the prime ideals 
$\fP$ above $\fp$ in $L$ satisfy $\frb{L/K}{\fP} \in D$. Then 
$S$ has Dirichlet density
$$ \delta(S) = \frac{\# D}{\# G}. $$
\end{thm}

The Frobenius density theorem is not the best we can hope for: 
in fact it does not even contain Dirichlet's theorem on primes
in arithmetic progression. Since the Frobenius automorphism of 
a prime ideal $\fp$ is determined up to conjugacy, the best possible 
density result would predict the infinitude of prime ideals whose 
Frobenius automorphism lies in some conjugacy class (in fact, 
divisions are unions of conjugacy classes). 

Frobenius' density theorem with divisions replaced by conjugacy classes
contains Dirichlet's Theorem \ref{TPAP} as a special case and was
stated as a conjecture by Frobenius at the end of his article \cite{Fro}.
For stating it, let $[\sigma]$ denote the conjugacy class of $\sigma \in G$:

\begin{thm}[Chebotarev's Density Theorem]\label{TCD}
Let $K/\Q$ be a normal extension, and fix a $\sigma \in G = \Gal(K/\Q)$.
Let $S$ denote the set of unramified primes $p$ with the property
that the prime ideals $\fp$ above $p$ in $K$ satisfy
$\frb{K/\Q}{\fp} \in [\sigma]$. Then $S$ has Dirichlet density
$$ \delta(S) = \frac{\# [\sigma]}{\# G}. $$
\end{thm}

The problem that Frobenius was unable to solve was that of resolving
the divisions into conjugacy classes. This was accomplished by Chebotarev
\cite{Cheb} using the technique of abelian twists. Already Hilbert, in his 
proof of the theorem of Kronecker and Weber, had introduced this technique: 
given a cyclic extension $K/\Q$, there is a cyclotomic extension $L/\Q$
such that the compositum $KL$ contains a cyclic extension $M/\Q$ with a
Galois group isomorphic to $K/\Q$ but with less ramification. 

\subsection*{Chebotarev and Artin}
Artin, who had conjectured his reciprocity law in 1923, had faced a 
problem similar to that of Frobenius. When he saw Chebotarev's article
\cite{Cheb}, Artin immediately suspected that this paper would hold the
key for a proof of his reciprocity law, and it did. In order to convince
the reader of the deep connection between the proofs, let us compare
the proofs of Chebotarev's density theorem as given by Ribenboim 
\cite[\S\ 25.3]{Rib} with the exposition of Artin's reciprocity law as 
given by Childress \cite{Child}:

\begin{table}[h!]
\begin{tabular}{|p{7.3cm}|l|l|} \hline 
       Step  &  Chebotarev & Artin \\ \hline
     1. The result holds if $K = \Q$ and $L/\Q$ 
     is a cyclotomic extension. This case follows from 
     Dirichlet's Theorem \ref{TPAP} (Chebotarev) and
     the irreduciblity of the cyclotomic equation (Dedekind;
     see the subsequent article \cite{Lem2}). 
     & \cite[p. 554]{Rib}       &    \cite[p. 112]{Child} \\
     2. The result holds for general base fields $K$ if $L = K(\zeta)$ is 
      a cyclotomic extension.
     & \cite[p. 554]{Rib}       &  \cite[p. 113]{Child} \\ 
     3. The result holds for general base fields $K$ if $L \subseteq K(\zeta)$ 
      is a subextension of a cyclotomic extension.
       & \cite[p. 558]{Rib}     & \cite[Ex. 5.6]{Child} \\ 
     4. The result holds for arbitrary cyclic extensions $L/K$. This step
      uses the technique of abelian twisting.
       &   \cite[p. 558]{Rib}   & \cite[p. 114 ff]{Child}  \\
     5. The result holds for general normal (Chebotarev) resp. abelian
        (Artin) extensions.
       & \cite[p. 561]{Rib}     &  \cite[p. 114]{Child} \\ \hline
   \end{tabular}\smallskip
\caption{Chebotarev's Density Theorem and Artin's Reciprocity Law}
\label{T1}
\end{table}

For proving step 4, Artin needed, in addition to Chebotarev's ideas,
auxiliary primes. These will be discussed in the next two sections.

In \cite{Art1}, where Artin conjectured his reciprocity law, he was
able to prove it for cyclic extensions of prime degree using the 
known reciprocity laws due to Kummer, Furtw\"angler and Takagi. Even
this proof already has a structure similar to the one above:
\begin{enumerate}
\item The Artin map sends the principal class (and only this class) to 
      the trivial automorphism.
\item If the reciprocity law holds for $K/k$, then it holds for ervery
      subextension $F/k$.
\item If the reciprocity law holds for $K_1/k$ and $K_2/k$, then it
      holds for the compositum $K_1K_2/k$.
\item The reciprocity law holds for cyclotomic extensions $K = k(\zeta)$.
\item If $k$ contains the $\ell^n$-th roots of unity, then the reciprocity
      law holds for cyclic extensions $K/k$ of degree $\ell^n$. Artin 
      derives this result from a Takagi's general reciprocity law, which
      was known to hold only for extensions of prime degree $\ell$, i.e.,
      for $n = 1$.
\item Let $K/k$ be a cyclic extension of prime power degree $\ell^n$,
      $\zeta$ a primitive $\ell^n$-th root of unity, and let 
      $k' = k(\zeta)$ and $K' = K(\zeta)$. If the reciprocity law holds
      for $K'/k'$, then it also holds for $K/k$.
\end{enumerate}

\section{Artin's Reciprocity Law}

The existence of the auxiliary primes necessary for Step 4 of Artin's 
proof is secured by the following lemma (Hilfssatz 1 in \cite{Art2}):

\begin{lem}
Let $f$ be a positive integer, and let $p_1$ and $p_2$ be primes.
Then there exist infinitely many primes $q$ with the following
property: the group $(\Z/q\Z)^\times$ has a subgroup $H$ such that
$p_1H = p_2H$, and the coset $p_1H$ has order divisible by $f$.
\end{lem}

Artin actually proved something stronger, namely that, in many
cases, the subgroup $H$ can be taken to be the group of $f$-th 
power residues modulo $q$. Below we only state this stronger result 
for $f = 2$ (I have modified the proof slightly in order to 
allow for the possibility $p_2 = -1$); observe that if $H$ is
the group of squares modulo $q$, then $pH$ has order $2$ if and only
if $(p/q) = -1$:

\begin{lem}
Let $p_1$ be a positive odd prime, and $p_2 \ne p_1$ a prime or 
$p_2 = -1$. Then there exist infinitely many primes $q$ with 
$(p_1/q) = (p_2/q) = -1$. 
\end{lem}

\begin{proof}
Consider the quadratic extension $F = \Q(\sqrt{p_1p_2}\,)$,
and let $K = F(\sqrt{p_1}\,)$. The Takagi group $T_{K/F}$ has 
index $2$ in the group $D$ of ideals coprime to the conductor of 
$K/F$, which divides $2p_1p_2$. The coset $D/T_{K/F}$ different
from $T_{K/F}$ contains infinitely many prime ideals of degree $1$; 
let $\fq$ be such a prime ideal coprime to $2p_1p_2$, and let $q$ 
be its norm.

Since $q$ splits in $F/\Q$, we must have $(p_1/q) = (p_2/q)$.
Since the prime ideals in $F$ above $q$ remain inert in $K/F$,
we must have $(p_1/q) = -1$. 
\end{proof}

Observe that Legendre's Lemma \ref{LLeg2} is exactly the case
$p_1 \equiv 1 \bmod 4$ and $p_2 = -1$ of this special case of 
Artin's Lemma. Perhaps the fact that an incarnation of Legendre's 
Lemma comes up as a special case in the proof of Artin's reciprocity
law shows that Legendre's work is much more than a failed attempt
of proving the quadratic reciprocity law.

\section{Arithmetization}

When Hasse \cite{HasAlg} later provided a proof of Artin's reciprocity 
law via the theory of algebras (see \cite{LR,FR}), he also had to prove 
the existence of certain auxiliary primes; his set of conditions 
is slightly different from Artin's:

\begin{lem}\label{THas}
If $p_1$, \ldots, $p_r$ are distinct prime numbers and $k_1$, \ldots, $k_r$
given natural numbers, then there is a modulus $m$ and a subgroup $U$
of the group $R = (\Z/m\Z)^\times$ such that 
\begin{enumerate}
\item $R/U$ is cyclic;
\item the order of residue class $p_j + m\Z$ in $R/U$ is a multiple of $k_j$
      for all $1 \le j \le r$;
\item the coset $-1 + m\Z$ has order $2$ in $R/U$.
\end{enumerate}
\end{lem}

Hasse remarked that the Frobenius density theorem guarantees the existence
of such a modulus $m$, and that it even can be chosen to be prime. He also
mentioned that he expected that this result can be proved with elementary
means by allowing $m$ to be composite, and that Artin meanwhile found
a proof of his reciprocity law by using the special case $r = 1$ of 
Hasse's Lemma \ref{THas}. 

There is a certain analogy with the following classical result in 
algebraic number theory: Dirichlet's analytic methods and class
field theory show that in a number field $K$, each ideal class in 
$\Cl(K)$ contains a prime ideal with degree $1$. Kummer, on the
other hand, observed that there is an algebraic proof of the 
slightly weaker fact that each ideal class contains an ideal whose
prime ideal factors all have degree $1$ (strictly speaking, the
result in this generality is due to Hilbert \cite[Satz 89]{Hil}; 
Kummer only had the general theory of ideal numbers in cyclotomic 
fields). For a similar approach to special cases of the Chebotarev
density theorem see Lenstra \& Stevenhagen \cite{SteLe}.

Chevalley \cite{Chev} succeeded in giving a non-analytic proof of the
main theorems of class field theory; in particular, he proved the
special case $r = 1$ of Hasse's Lemma \ref{THas}, which 
Hasse  presented (with a simplified proof) in his lectures 
\cite[Satz 139]{HasKKT} on class field theory:

\begin{thm}
Let $a > 1$, $k$ and $n$ be given integers. Then there exists a modulus
$m$ coprime to $k$ for which $R = (\Z/m\Z)^\times$ contains a subgroup 
$U$ with the following properties:
\begin{enumerate}
\item $R/U$ is cyclic;
\item the order of the coset $aU$ is divisible by $n$.
\end{enumerate}
\end{thm}

If $a$ is an odd prime, and if we demand that $m$ be prime and that $U$ 
is the subgroup of squares modulo $m$, then the case $n = 2$ of this 
theorem predicts the existence of a prime $m$ such that the coset 
$a + m\Z$ has order $2$ modulo squares, i.e., that $(\frac ma) = -1$.
Showing the existence of primes in certain classes seems to require
analytic techniques in most cases; Chevalley's success in giving 
an arithmetic proof of class field theory is due in part to Hasse's 
insight that Artin's proof can be modified in such a way that one 
can do with nonprime moduli $m$. An elementary proof of 
Lemma \ref{THas} for general $r$ was given by van der Waerden 
\cite{vdW}.

As we have seen, an important step in the arithmetization of 
class field theory was the realization that the role of 
auxiliary primes could be played by composite numbers with
suitable properties, whose existence could be proved without
Dirichlet's analytic techniques. This brings up the question
whether Legendre's proof of the quadratic reciprocity law or
Gauss's first proof can be modified in a similar way.

\vskip 0,5cm
\end{document}